\documentclass[10pt]{amsart}
\usepackage{tikz,array, verbatim}
\usepackage{amsfonts, amsmath, latexsym, epsfig, caption}
\usepackage{amssymb, color}
\usepackage{epsf}
\usepackage{array}
\usepackage{ragged2e}
\usepackage{hyperref}
\DeclareMathOperator{\Stab}{Stab}

\DeclareMathOperator{\tcvp}{cvp}
\DeclareMathOperator{\tCVP}{CVP}

\DeclareMathOperator{\Del}{Del}
\DeclareMathOperator{\parity}{Parity}
\usepackage{mathrsfs}

\newcommand{\conv}{\operatorname{conv}}

\newcommand{\GL}{\operatorname{{\mathbf GL}}}

\newcommand{\cD}{{\mathcal D}}

\newcommand{\cS}{{\mathcal S}}

\newcommand{\R}{{\mathbb R}}

\newcommand{\Z}{{\mathbb Z}}

\newcommand{\F}{{\mathbb F}}

\title{Iso Edge Domains}

\def\QuotS#1#2{\leavevmode\kern-.0em\raise.2ex\hbox{$#1$}\kern-.1em/\kern-.1em\lower.25ex\hbox{$#2$}}

\begin{document}


\author[M. D. Sikiri\'c]{Mathieu Dutour Sikiri\'c}
\address{Mathieu Dutour Sikiri\'c, Rudjer Boskovi\'c Institute, Bijeni\v{c}ka 54, 10000 Zagreb, Croatia}
\email{mathieu.dutour@gmail.com}

\author[M. Kummer]{Mario Kummer}
\address{Technische Universit\"at Dresden,
 Fakult\"at Mathematik,
 Institut f\"ur Geometrie,
 Zellescher Weg 12-14,
 01062 Dresden, Germany}
\email{mario.kummer@tu-dresden.de}

\newcommand{\RR}{\ensuremath{\mathbb{R}}}
\newcommand{\NN}{\ensuremath{\mathbb{N}}}
\newcommand{\QQ}{\ensuremath{\mathbb{Q}}}
\newcommand{\CC}{\ensuremath{\mathbb{C}}}
\newcommand{\ZZ}{\ensuremath{\mathbb{Z}}}
\newcommand{\TT}{\ensuremath{\mathbb{T}}}

\newtheorem{theorem}{Theorem}[section]
\newtheorem{proposition}[theorem]{Proposition}
\newtheorem{corollary}[theorem]{Corollary}
\newtheorem{lemma}[theorem]{Lemma}
\newtheorem{problem}[theorem]{Problem}
\newtheorem{conjecture}{Conjecture}
\newtheorem{question}{Question}
\newtheorem{claim}{Claim}
\newtheorem{remark}[theorem]{Remark}
\theoremstyle{definition}
\newtheorem{definition}[theorem]{Definition}
\newtheorem{ex}[theorem]{Example}

\begin{abstract}
Iso-edge domains are a variant of the iso-Delaunay decomposition introduced by Voronoi.
They were introduced by Baranovskii \& Ryshkov in order to solve the covering problem in
dimension $5$.

In this work we revisit this decomposition and prove the following new results:
\begin{enumerate}
\item We review the existing theory and give a general mass-formula for the iso-edge domains.
\item We prove that the associated toroidal compactification of the moduli space of principally polarized abelian varieties is projective.
\item We prove the Conway--Sloane conjecture in dimension $5$.
\item We prove that the quadratic forms for which the conorms are non-negative are
  exactly the matroidal ones in dimension $5$.
\end{enumerate}

\end{abstract}

\maketitle

\section{Introduction}\label{SectionIntroduction}
Voronoi introduced several {polyhedral} decompositions of the cone of positive definite quadratic forms {that are invariant under the action of the general linear group over the integers}.
The decomposition \cite{VoronoiI} uses perfect forms while the decomposition \cite{VoronoiII}
uses Delaunay polytopes. See \cite{AchillBook} for an overview of such decompositions.

Let $A\in\cS^n_{>0}$ be a positive definite quadratic form on $\R^n$. {A polytope $P\subset\R^n$ with vertices in $\Z^n$}
is called a \emph{Delaunay polytope} if there is some $c\in\R^n$ and $r>0$ such that $${A[x-c]:=(x-c)^tA(x-c)\geq r}$$
for all $x\in\Z^n$ with equality if and only if $x$ is a vertex of $P$. {It follows from the definition that if $P$ is a Delaunay polytope and $x\in\Z^n$, then $x+P$ is also a Delaunay polytope.} The set $\Del(A)$ of all
Delaunay polytopes is the \emph{Delaunay subdivision} associated to $A$. The \emph{secondary cone} of a
Delaunay subdivision $\cD$ is $$\Delta(\cD)=\{A'\in\cS^n_{>0}:\, \Del(A')=\cD\},$$also known as
\emph{$L$-type domain}. The secondary cone $\Delta(A)$ of $A\in\cS^n_{>0}$ is defined to
be $\Delta(\Del(A))$. Voronoi's second reduction theory \cite{VoronoiII} states that the set of
all secondary cones is {an open} polyhedral subdivision of the cone $\cS^n_{>0}$ of positive definite
symmetric matrices on which the group $\GL_n(\Z)$ acts by conjugation. For any fixed $n$ this
group action has only finitely many orbits. The \emph{iso-edge domain}, or \emph{$C$-type domain}\footnote{The terminology ``$C$-type'' comes from ``$L$-type'' introduced by Voronoi. The Cyrillic
  letter $C$ is pronounced like the Latin letter $S$ and refers to ``skeleton'', that is the graph
  determined by the edges of the Delaunay subdivision.},
introduced in \cite{InitialCtypeRyshkovBaranovskii} is a coarser subdivision of $\cS^n_{>0}$:
The iso-edge domain of a Delaunay subdivision $\cD$ is defined as
\begin{equation*}
C(\cD)=\{Q\in\cS^n_{>0}:\, E\in\Del(Q)\textrm{ for all }E\in\cD\textrm{ with }E \textrm{ centrally symmetric}\}.
\end{equation*}
{For a generic positive definite quadratic form, all Delaunay polytopes are simplices. Thus, in this case,} the only centrally symmetric Delaunay polytopes are the edges justifying the name.
The iso-edge domain $C(A)$ of $A\in\cS^n_{>0}$ is defined to be $C(\Del(A))$.
It was shown in \cite{InitialCtypeRyshkovBaranovskii} that the set of all iso-edge domains
is a tiling of $\cS^n_{>0}$.
Each iso-edge domain is the union of finitely many secondary cones.
Moreover, in each fixed dimension $d$ there are {only} finitely many iso-edge
domains up to $\GL_n(\Z)$-equivalence. {Further, the authors of \cite{InitialCtypeRyshkovBaranovskii}} classified the iso-edge domains in dimension $5$ and found
$76$ full-dimensional types. This was {confirmed} in \cite{EngelRediscovery}.
Recently, in \cite{CtypeDimSix} the classification in dimension $6$ yielded 55083358
full-dimensional types.

In Subsection \ref{SUBSEC_ParityVector} we give general definitions {of} the domains.
In Subsection \ref{SUBSEC_EnumerationCtype} an algorithm for enumerating the iso-edge domains
in a fixed dimension is given.

In {algebraic geometry}, compactifications of the space of principally polarized abelian varieties
were considered and they are described by equivariant polyhedral {decompositions} of the cone of positive
definite quadratic forms. The most thoroughly studied compactifictions are the perfect form decomposition,
the central cone compactification and the secondary cone compactification (see \cite{NamikawaToroidal}).
In Section \ref{SectionCompactification} we prove that the iso-edge domains define a projective
compactification.
In Subsection \ref{SUBSEC_MassFormula} a mass formula involving all the iso-edge domains of a given
dimension is given.
In Subsection \ref{SUBSEC_EnumerationDimFive} the enumeration of all iso-edge domains in dimension $5$
is presented.

The notion of iso-edge domain is useful if one is interested in properties that only depend on the edges of
the Delaunay polytope or, equivalently, on the facets of the Voronoi polytope,
see e.g. \cite{ChromaticVoronoi}.
We address here two questions of this kind. For every {edge} $e = [u,v]$ of a Delaunay polytope,
the {mid point} $(u+v)/2$ defines a class $c$ in $\frac{1}{2}L / L$ where $L$ is the lattice.
We then set $f(c) = \Vert u - v\Vert$. Conway \& Sloane conjectured that this function
characterizes $L$ up to conjugacy and prove it in dimension at most $3$ \cite[p.~58]{conwaysloanevi}.
The conjecture is proved in dimension $4$ in \cite[Ch 4.4]{VallentinThesis}.
We prove this conjecture in dimension $5$ in Section \ref{SectionConwaySloane}.
We further prove in dimension $4$ and {in dimension} $5$ that the {Fourier transform of $f$ takes only non-negative values} if and only
if $A$ belongs to the matroidal locus. This implies a positive answer to \cite[Question 4.11]{ChuaKummerSturmfels}
for $d\leq 5$.


In the final Section \ref{SectionGeneralization} we list a number of possible generalizations of
the iso-edge domains that could be of interest to researchers.

\section{Definitions and Enumeration techniques}\label{SectionEnumeration}

\subsection{Voronoi polytopes and parity vectors}\label{SUBSEC_ParityVector}
For any natural number $n$ we define $$\parity_n=\left\{0,\frac{1}{2}\right\}^n - \{0\}$$ the set of parity vectors.
Given a positive definite quadratic form $A$ and a vector $v$ we define $A\left\lbrack v\right\rbrack = v^t A v$ and
\begin{equation*}
\begin{array}{rcl}
\tcvp(A,v) &=& \min_{x\in \ZZ^n} A\left\lbrack x - v\right\rbrack\\
\tCVP(A,v) &=& \left\{x\in \ZZ^n\mbox{~s.t.~}  A\left\lbrack x - v\right\rbrack = \tcvp(A,v)\right\}
\end{array}
\end{equation*}
{The notations $\tcvp$ and $\tCVP$ refer to the \emph{closest vector problem}.}

{
\begin{lemma}
 The mapping $$v \mapsto\conv( \tCVP(A,v))$$ establishes a bijection between $\parity_n$ and the set of centrally symmetric Delaunay polytopes (up to translation by some $w\in\Z^n$).
\end{lemma}

\begin{proof}
 Let $P=\conv( \tCVP(A,v))$. It follows directly from the definition that $P$ is a Delaunay polytope. We claim that it is centrally symmetric with center $v$. Indeed, let $x\in\tCVP(A,v)$. Then the reflection of $x$ across $v$ is $2v-x\in\Z^n$. We have $A[(2v-x)-v)]=A[v-x]=\tcvp(A,v)$ which shows that $2v-x\in\tCVP(A,v)$. Thus $P$ is a centrally symmetric Delaunay polytope.
 
 Conversely, let $P$ be a centrally symmetric Delaunay polytope with center $v$. For any vertex $x\in\Z^n$ of $P$, the reflection $y$ of $x$ across $v$ is also a vertex of $P$ and thus $y\in\Z^n$. Since $v$ is the mid point of $[x,y]$, we can write $v=v'+a$ for $v'\in\parity_n$ and $a\in\Z^n$. After a translation we can thus assume that the center $v$ of $P$ is in $\parity_n$. Mapping $P$ to $v\in\parity_n$ thus defines the inverse map of $v \mapsto\conv( \tCVP(A,v))$.
\end{proof}

}

In particular, the edges of Delaunay polytopes determine centrally symmetric Delaunay polytopes.
For a positive definite quadratic form $A$ the \emph{Voronoi polytope} is defined as
\begin{equation*}
Vor(A) = \left\{ x\in \RR^n \mbox{~s.t.~} A[x] \leq A[x - v] \mbox{~for~} v\in \ZZ^n - \{0\} \right\}.
\end{equation*}
The facets of the {Voronoi} polytope correspond to the edges of the Delaunay polytope and there are at most $2(2^n-1)$ of them.

The iso-edge domain of a Delaunay subdivision is
\begin{equation*}
C(\cD) = \left\{Q\in\cS^n_{>0}:\, D \in \Del(Q)\textrm{ for all }D\in\cD\textrm{ with }D \textrm{ centrally symmetric}\right\}.
\end{equation*}
{If} the number of facets of the Voronoi polytope is $2(2^n-1)$, {then} the iso-edge domain is full dimensional and
we say that it is {\em primitive}.
This is equivalent to saying that for each $v\in \parity_n$ we have $\left\vert\tCVP(A,v)\right\vert = 2$
or to saying that all the centrally symmetric Delaunay polytopes are $1$-dimensional.

\begin{ex}\label{ex:dim4}
 In dimension $n=4$ there are three non-equivalent full-dimensional secondary cones. In \cite[\S4.4.1]{VallentinThesis} an explicit description of one representative for each is given by means of extreme rays. Each of these secondary cones is simplicial. According to \cite{InitialCtypeRyshkovBaranovskii} the number of non-equivalent full-dimensional iso-edge domains is three as well. However, not all of them are simplicial. Indeed, consider the secondary cone $\Delta(\cD_3)$ from \cite[\S4.4.1]{VallentinThesis}. Its extreme rays are the rank one matrices $v_iv_i^t$ for $i=1,\ldots,9$ where $$V=\begin{pmatrix}|&&|\\ v_1 & \cdots & v_9 \\ |&&|\end{pmatrix}=\begin{pmatrix}1&0&0&0&1&1&0&0&1\\ 0&1&0&0&0&0&1&1&1\\ 0&0&1&0&-1&0&-1&0&-1\\ 0&0&0&1&0&-1&0&-1&-1 \end{pmatrix}$$and the rigid positive definite matrix $$R=\begin{pmatrix}4&2&-2&-2\\ 2&4&-2&-2\\ -2&-2&4&0\\ -2&-2&0&4\end{pmatrix}.$$The edges of $\cD_3$ that are adjacent to the origin are given by the columns of the matrix \[E=\begin{pmatrix}1&0&0&0&1&1&1&0&0&0&1&1&1&0&1\\ 0&1&0&0&-1&0&0&1&1&0&1&1&0&1&1\\ 0&0&1&0&0&1&0&1&0&1&1&0&1&1&1\\ 0&0&0&1&0&0&1&0&1&-1&0&1&1&1&1\end{pmatrix}\]
 and their negatives, see for example \cite[\S8.4.2]{VallentinThesis}. Multiplying the matrix $$S=\begin{pmatrix}0&0&1&0\\0&0&0&1\\1&0&0&0\\0&1&0&0\end{pmatrix}$$from the left to $V$ or $E$ permutes the columns (up to sign). This implies that the secondary cone $S^t\Delta(\cD_3)S$ belongs to the same iso-edge domain as $\Delta(\cD_3)$. Furthermore, the two equivalent simplicial secondary cones $\Delta(\cD_3)$ and $S^t\Delta(\cD_3)S$ share the facet whose extreme rays are the $v_iv_i^t$ for $i=1,\ldots,9$. Finally, these two secondary cones are the only two full-dimensional secondary cones in their iso-edge domain because they do not have any other bistellar neighbours that are equivalent to $\Delta(\cD_3)$ \cite[\S4.4.1]{VallentinThesis}. This shows that this iso-edge domain is the non-simplicial cone whose extreme rays are the $v_iv_i^t$ for $i=1,\ldots,9$ together with the two rigid positive definite matrices $R$ and $S^tRS$.
\end{ex}

\subsection{Enumeration of full dimensional iso-edge domains}\label{SUBSEC_EnumerationCtype}
A typical strategy used for enumerating polyhedral {subdivisions} such as perfect domains \cite{VoronoiI}
or secondary cones \cite{VoronoiII} is a graph traversal algorithm: One starts from one cone and for
each cone one computes the adjacent cones and {inserts} them in the complete list of cones if they are
not isomorphic to already known cones. The enumeration ends when all cones have been treated.

{A well known full-dimensional Delaunay subdivision is the one associated to the quadratic form known as \emph{Voronoi's principal form of the first type}: $$A[x]=n\sum_{i=1}^nx_i^2-\sum_{1\leq i<j\leq n}x_ix_j. $$The Delaunay subdivision of $A$ is a triangulation, see for example \cite[\S2.3]{VallentinThesis}. In particular, every centrally symmetric Delaunay polytope is $1$-dimensional. Thus by \cite[Theorem 16]{EngelRediscovery} it} defines a primitive iso-edge domain which we use as starting cone of the enumeration.

For a primitive iso-edge domain $D$ formed by vectors $(v_i)_{1\leq i \leq N}$ with $N = 2(2^n - 1)$ we determine the
finite set of triples $\{i, j, k\}$ such that $v_i + v_j + v_k = 0$.
{For any such triple, the triangle inequality gives three inequalities: $$A[v_i]=A[v_j+v_k] \leq A[v_j] + A[v_k]$$}and its permutations.
The finite set of all such inequalities is denoted by ${\mathcal F}$ and defines $D$.
From ${\mathcal F}$ we can determine by linear programming which inequalities determine a facet of $D$.
Thus we have a set of {facet-defining} inequalities as $(\phi_i)_{i\leq i\leq m}$ and we need to find the adjacent
full-dimensional iso-edge domain.
Given a facet defining inequality $\phi(A)\geq 0$, it is proportional to a number of inequalities
of the form $A[v_{i_u}] \leq A[v_{j_u}] + A[v_{k_u}]$ for $1\leq u\leq r$.
In the adjacent domain the vectors $v_{i_u}$ are replaced by $v_{j_u} - v_{k_u}$ and this defines the adjacent
domain. {Further note that the center of a one-dimensional Delaunay polytope becomes the center of a two-dimensional Delaunay polytope of a quadratic form in the lower dimensional iso-edge domain that separates two adjacent full dimensional domains.}

A full-dimensional iso-edge domain is encoded by a finite set of vectors $(v_i)_{1\leq i\leq N}$ with $N=2(2^n - 1)$
and we consider equivalence under $\GL_n(\ZZ)$.
We can thus apply the methodology of \cite[Section 3.1]{LMS_JCM_groupPolytope} in order to test
isomorphy of full-dimensional iso-edge domains.

\section{Compactification}\label{SectionCompactification}

We define $S^n_{rat, \geq 0}$ to be the {\em rational closure} of $S^n_{>0}$, that is the positive semidefinite matrices
whose kernel is defined by rational equations.
A polyhedral decomposition of $S^n_{\geq 0}$ with cones $\sigma_{\alpha}$ is called
$\GL_n(\ZZ)$-admissible (see \cite[II.4, Definition 4.10]{BookAMRT}) if
\begin{enumerate}
\item[(A1)] Each face of a $\sigma_{\alpha}$ is a $\sigma_{\beta}$.
\item[(A2)] $\sigma_{\alpha} \cap \sigma_{\beta}$ is a common face of $\sigma_{\alpha}$ and $\sigma_{\beta}$.
\item[(A3)] $\gamma \sigma_{\alpha}$ is a $\sigma_{\beta}$ for all $\gamma \in \GL_n(\ZZ)$.
\item[(A4)] modulo $\GL_n(\ZZ)$ there are only a finite number of $\sigma_{\beta}$
\item[(A5)] $S^n_{rat, \geq 0} = \cup_{\alpha} (\sigma_{\alpha} \cap S^n_{rat, \geq 0})$
\end{enumerate}

An admissible decomposition defines a compactification of the cone of principally polarized
abelian varieties (see \cite[III.5, Theorem 5.2]{BookAMRT}).
According to \cite[IV.2. Definition 2.1]{BookAMRT} an admissible decomposition defines
a projective variety if there exists a function $\phi:S^n_{rat, \geq 0}{\to\R}$ such that:
\begin{enumerate}
\item[(B1)] $\phi$ is convex, piecewise-linear
\item[(B2)] $\phi(A) > 0$ for $A\not= 0$
\item[(B3)] $\phi$ is linear on {each $\sigma_{\alpha}$ and $\phi$ is not linear on any strict superset of any $\sigma_{\alpha}$.}
\item[(B4)] $\phi$ is integral on $\Sigma_{\ZZ}$ which in this case are the integral valued matrices.
\end{enumerate}

\begin{lemma}\label{lem:injective}
Let $A$ be a quadratic form on $\R^n$. If for each coset $\alpha\in\Z^n/2\Z^n$ there is an $x_{\alpha}\in \alpha$ with $A[x_{\alpha}]=0$, then $A=0$.
\end{lemma}

\begin{proof}
 Assume for the sake of a contradiction that the claim is not true. This means that the vector space $N$ of all quadratic forms $A$ with $A[x_\alpha]=0$ for all $\alpha\in\Z^n/2\Z^n$ has positive dimension. Since $x_\alpha\in\Z^n$, the vector space $N$ must contain a nonzero quadratic form all of whose coefficients are rational numbers. After scaling appropriately, we obtain a quadratic form $A\in N$ all of whose coefficients are integers such that at least one coefficient is odd. Reducing modulo $2$ gives a quadratic form $Q$ on $\F_2^n$ with at least one nonzero coefficient which vanishes on all of $\F_2^n$.
 This cannot happen as $Q[\delta_i]=0$ for $i=1,\ldots,n$ forces the coefficient of $x_i^2$ to be zero and then $Q[\delta_i+\delta_j]=0$ forces the coefficient of $x_ix_j$ to be zero.
\end{proof}

\begin{theorem}\label{thm:admissible}
The iso-edge domains form a $\GL_n(\ZZ)$-admissible decomposition of $S^n_{rat, \geq 0}$
and define a projective compactification.
\end{theorem}
\proof We first define the function $\phi$ as
\begin{equation*}
\phi(A) = \sum_{x\in \parity_n} 4 \tcvp(A, x).
\end{equation*}
The convexity and positivity is clear. Lemma \ref{lem:injective} implies that $A=0$ if $\phi(A)=0$ proving (B2).

On a given full dimensional iso-edge domain, $\phi$ is linear proving (B1).
Let us consider two adjacent full dimensional iso-edge domains $C_1$ and $C_2$.
If an edge vector $v$ in $C_1$ becomes a vector $v'$ in $C_2$ then there exist a triple $(v, v_1, v_2)$ in $C_1$
with $v + v_1 + v_2 = 0$ such that $v' = v_2 - v_1$. We then have the relation $A[v] + A[v'] = 2 A[v_1] + 2 A[v_2]$.
Thus we have on the interior of $C_1$ the relation $A[v] < A[v']$ and the reverse on $C_2$.
This shows that the maximal cone in which $\phi$ is linear does not extend beyond $C_1$
and so is exactly $C_1$ proving (B3).

For an integral matrix and a parity vector $x$ the minimum $\tcvp(A, x)$ is realized by a vector $v\in \ZZ^n$
thus $4 \tcvp(A,x) = A[2x - v]\in \ZZ$ proving (B4). Therefore we have a characteristic function $\phi$.
\smallskip

The iso-edge {domains} define convex polyhedral cones $\sigma_{\alpha}$.
Taking a face of such a polyhedral cone corresponds to extending the set of closest {points} proving (A1).\\

Let us consider the convex set ${\mathcal C}$ defined by $\phi(A) \geq 1$. The faces $\sigma_{\alpha}$
of this cone are defined by $\ell(x)=1$ for $\ell$ a function such that $\ell(x)\geq 1$ for $x\in {\mathcal C}$.
If $\sigma_{\alpha}$ and $\sigma_{\beta}$ are two cones of function $\ell_{\alpha}$ and $\ell_{\beta}$
then $\sigma_{\alpha}\cap \sigma_{\beta}$ is a cone of the polyhedral tessellation for the function
$(\ell_{\alpha} + \ell_{\beta}) / 2$ proving (A2).

The equivariance (A3) of iso-edge domains follows from the fact that the definitions of inequalities are invariant
under $\GL_n(\ZZ)$ equivalence.
Since every secondary cone belongs to an iso-edge domain and each iso-edge domain contains finitely many
secondary cones, finiteness (A4) follows.
Since every $A\in S^n_{rat, \geq 0}$ belongs to a secondary cone, it belongs to an iso-edge domain.
Conversely, since every iso-edge domain is the union of secondary cones,
it is included in $S^n_{rat, \geq 0}$ proving (A5). \qed

{
Already for $n=4$ this compactification is singular. Indeed, by Example \ref{ex:dim4} there is a non-simplicial iso-edge domain and non-simplicial cones give rise to singularities of the compactification by \cite[Ch.I, \S1, Thm.~4]{TEI}.}

\section{Lower dimensional cells of iso-edge domain and their mass formula}

\subsection{Mass formula}\label{SUBSEC_MassFormula}

We follow here the methodology of \cite{DutourGunnells} where mass formulas are established in the case of
perfect domains over imaginary fields.

\begin{theorem}\label{Mass_Formula}
Let us take $n\geq 3$. Let us take the set ${\mathcal S}$ of all iso-edge domains which contain a positive definite matrix. {The group $\GL(\Z^n)$ acts on $\mathcal{S}$ and for any $S\in\mathcal{S}$ we denote by $\Stab(S)$ its stabilizer.} Then we have the formula
\begin{equation}\label{Equa_Mass_Formula}
\sum_{S\in {\mathcal S}} \frac{(-1)^{\dim S}}{\left\vert \Stab(S) \right\vert} = 0.
\end{equation}

\end{theorem}
\proof For each full dimensional iso-edge domain $S$ we compute the list of extreme rays.
For each extreme ray we choose a generator with integral coefficients whose greatest common divisor is $1$.
To $S$ we associate the sum $iso(S)$ of those generators which are invariant under the stabilizer of $S$.

A given iso-edge domain containing a positive definite quadratic form is contained in a finite set of
full-dimensional iso-edge domains $S_1$, \dots, $S_L$. We encode the set {by} $$\{iso(S_1), \dots, iso(S_L)\}.$$
This defines a complex which is contractible since the cone of positive definite quadratic form is
contractible.

The proof then follows the same line as the one of \cite[Theorem 4.6]{DutourGunnells} and the conclusion
comes from the fact that the Euler Poincar\'e characteristic of $\GL_n(\ZZ)$ is $0$ for $n\geq 3$
(see \cite{HarderGaussBonnet}). \qed

\subsection{Full decomposition for $n=5$}\label{SUBSEC_EnumerationDimFive}
The classification of possible Delaunay tessellations in dimension $5$ was obtained
in \cite{ClassifDimFive}.
The full dimensional iso-edge domains were determined in \cite{InitialCtypeRyshkovBaranovskii}.
In order to determine the iso-edge domain of lower dimension, we compute the facets of the top-dimensional
ones and check for isomorphism. Then we reiterate the procedure until we reach the iso-edge domains
of dimension $1$. We denote by $\sigma_k^n$ the set of $k$-dimensional cells in dimension $n$.
The classification for $n=5$ is given in Table \ref{tab:TableCtype_g5}.
As in \cite{ClassifDimFive} a significant check for correctness comes from Theorem \ref{Mass_Formula}.
{The stabilizers are computed in the following way: For each iso-edge domain $S$, we compute the list of
extreme rays $\{e_1, \dots, e_m\}$ and for each extreme ray $e_i$, we compute an integer matrix
generator $q_i$ that we normalize by taking the greatest common divisor of their coefficients to
be one. We then define the quadratic form that is $q(S)= \sum_i q_i$. We then compute the arithmetic
automorphism group of this positive definite form using methods of \cite[Sec. 3.1.6]{AchillBook}.}

\begin{table}
\caption{Number of $\GL_5(\ZZ)$-inequivalent iso-edge cones in~${\mathcal S}^5_{>0}$ by their dimension.}
\label{tab:TableCtype_g5}
\begin{center}
\begin{tabular}{|c|c||c|c||c|c|}
\hline
dim & nr. iso. edge      & dim & nr. iso. edge      & dim & nr. iso. edge\\
\hline
1 & 5          & 6 & 2478           & 11 & 8796\\
2 & 24         & 7 & 5180           & 12 & 4905\\
3 & 90         & 8 & 8642           & 13 & 1927\\
4 & 318        & 9 & 11350          & 14 & 478\\
5 & 972        & 10 & 11472         & 15 & 76\\
\hline
\end{tabular}
\end{center}
\end{table}

\section{The Conway--Sloane conjecture}\label{SectionConwaySloane}

For any $v\in\parity_n$ we call $\Theta_v(A):=-\tcvp(A,v)$ the \emph{tropical theta constant with characteristic $v$} because it describes the asymptotic behavior of the classical Riemann theta constants. Note that the \emph{vonorm} of the class $2v+2\Z^n$ as defined by Conway and Sloane in \cite[\S3]{conwaysloanevi} equals $-4\cdot\Theta_v(A)$. They state that $A$ is determined up to $\GL_n(\Z)$-equivalence by the vonorms for $n\leq4$ and conjecture that this holds true for arbitrary $n$:
\begin{conjecture}[Conway--Sloane]\label{con:conslo}
The map $\Theta: S^n_{>0}\to\R^{2^n-1}$ defined by
\begin{equation*}
\Theta(A) = \left(\Theta_v(A)\right)_{v\in \parity_n}
\end{equation*}
is injective on any fundamental domain of the $\GL_n(\Z)$-action on $S^n_{>0}$.
\end{conjecture}
Another way to phrase Conjecture \ref{con:conslo} is that a tropical abelian variety is determined by its theta constants. This gives positive evidence for the conjecture as the corresponding statement for classical abelian varieties is true \cite{tata3}.
We will prove that it is true for $n=5$ as well as when restricted to the matroidal locus.
Lemma \ref{lem:injective} shows injectivity on each iso-edge domain.

\begin{lemma}\label{lem:segment}
 Let $A\in S_{>0}^n$ and $x,y\in\Z^n$. The line segment {$[x,y]$} is in $\Del(A)$ if and only if the closed ball $B$ (with respect to $A$) around $\frac{x+y}{2}$ with radius $r=\left\Vert\frac{x-y}{2}\right\Vert$ contains no lattice points other than $x$ and $y$.
\end{lemma}

\begin{proof}
 The ``if'' direction is clear from the definition. In order to show the other direction, let $z\in B\cap\Z^n$ with $z\not\in\{x,y\}$. Note that $z'=(x+y)-z\in B\cap\Z^n$. Further let $w\in\R^n$ such that $A[w-x]=A[w-y]$. We will show that $\min(A[w-z],A[w-z'])\leq A[w-x]$ which shows the claim. To this end, we first observe that $A[w-x]=A[w-y]$ implies $(\frac{x-y}{2})^t\cdot A\cdot (w-\frac{x+y}{2})=0$. We further have that
 \begin{equation*}
\begin{array}{rcl}
A\left\lbrack w-z \right\rbrack &=& A\left\lbrack w-\frac{x+y}{2} \right\rbrack + A\left\lbrack z-\frac{x+y}{2} \right\rbrack - (w-\frac{x+y}{2})^t A(z-\frac{x+y}{2}),\\
A\left\lbrack w-z'\right\rbrack &=& A\left\lbrack w-\frac{x+y}{2} \right\rbrack + A\left\lbrack z-\frac{x+y}{2} \right\rbrack + (w-\frac{x+y}{2})^t A(z-\frac{x+y}{2}).
\end{array}
\end{equation*}
 Therefore, we have
\begin{equation*}
\begin{array}{rcl}
\min(A\left\lbrack w-z\right\rbrack , A\left\lbrack w-z'\right\rbrack )
&\leq& A\left\lbrack w-\frac{x+y}{2}\right\rbrack + A\left\lbrack z-\frac{x+y}{2}\right\rbrack\\
&\leq& A\left\lbrack w-\frac{x+y}{2}\right\rbrack + A\left\lbrack \frac{x-y}{2}\right\rbrack = A\left\lbrack w-x\right\rbrack.
\end{array}\qedhere
\end{equation*}
\end{proof}

\begin{lemma}\label{lem:linear}
 The map $\Theta: S^n_{>0}\to\R^{2^n-1}$ is linear and injective on every iso-edge domain.
\end{lemma}

\begin{proof}
 We show that $\Theta$ is linear and injective on the closure of every full-dimensional iso-edge domain. This implies the claim, see also Section \ref{SectionCompactification}.
 Let $A$ be a positive definite $n\times n$ matrix and ${c}\in\parity_n$. We first assume that $\Del(A)$ is a triangulation. Let $v_{{c}}\in\Z^n$ be a closest (with respect to $A$) lattice point to ${{c}}$ and define $r=A[{{c}}-v_{{c}}]$. We clearly have $r=A[(2{{c}}-v_{{c}})-{{c}}]$. Thus $2{{c}}-v_{{c}}$ is also a closest lattice point to ${{c}}$. This means that $v_{{c}}$ and $2{{c}}-v_{{c}}$ are vertices of a Delaunay polytope of $A$. Since $\Del(A)$ is a triangulation, this shows that the line segment from $v_{{c}}$ to $2{{c}}-v_{{c}}$ is in $\Del(A)$. Thus it follows from Lemma \ref{lem:segment} that $v_{{c}}$ is a closest lattice point to ${{c}}$ with respect to every $A'$ in the iso-edge domain of $A$. This means that $$\tcvp(A',{{c}}) = \min_{x\in \ZZ^n} A'[x-{{c}}]=A'[v_{{c}}-{{c}}]$$for all $A'\in C(A)$ which shows that $\Theta$ is linear on $C(A)$. Since both sides of the above equation are continuous in $A'$, this equality holds on the closure of the iso-edge domain as well. Finally, assume that $\Theta(A)=\Theta(A')$ for some $A'\in\overline{C(A)}$. Then we have $(A'-A)[v_{{c}}-{{c}}]=0$ for every ${{c}}\in\parity_n$. Thus by Lemma \ref{lem:injective} we have $A'-A=0$.
\end{proof}

Thus in order to prove the Conway--Sloane conjecture it remains to show that the image under $\Theta$ of any two different representatives of $\GL_n(\Z)$-equivalence classes of iso-edge domains are disjoint. We carry this out for $n=5$.

\begin{theorem}
The Conway--Sloane conjecture is true in dimension $5$.
\end{theorem}
\begin{proof} On a given iso-edge domain the map $\Theta$ is a linear map. The image of the polyhedral cone
of the iso-edge domain is a polyhedral cone.

Since we know the iso-edge domains in dimension $5$, we can look at the obtained $76$ different
images. We found out the following for each pair $C_1$, $C_2$.
\begin{enumerate}
\item The pairwise intersection $\Theta(C_1)\cap \Theta(C_2)$ does not {happen} in their relative interiors.
\item The pairwise intersection $\Theta(C_1)\cap \Theta(C_2)$ {defines} a face of {$\Theta(C_1)$} and $\Theta(C_2)$.
\item The corresponding faces of $C_1$ and $C_2$ are arithmetically equivalent.
\end{enumerate}
Thus if two matrices have the same image then they are arithmetically equivalent. \end{proof}

\section{The matroidal locus}\label{SectionMatroidalLocus}
{Recall that a matrix with integer entries is \emph{totally unimodular} if every square submatrix has determinant $0$, $+1$ or $-1$.} Let $M$ be a {totally} unimodular $n\times r$ matrix with columns $v_1,\ldots,v_r\in\Z^n$. The set of positive definite matrices in the conic hull of the rank one matrices $v_1v_1^t,\ldots,v_r v_r^t$ is denoted by $\sigma(M)$. The union of all $\sigma(M)$ for some $r$ and some {totally} unimodular $n\times r$ matrix $M\in\Z^{n\times r}$ is called the \emph{matroidal locus} and is studied in \cite{Melo_Viviani}. In order to study the matroidal locus, we consider the Fourier transforms of the tropical theta constants.

\begin{definition}
 Let $A\in S^n_{>0}$ and $v\in\parity_n$. Then we define $$\vartheta_v(A)=\frac{1}{2^{n-3}}\sum_{w\in \{0,1/2\}^n} (-1)^{4\cdot v^{{t}} w}\cdot \Theta_w(A).$$Following \cite{conwaysloanevi} we call $\vartheta_v(A)$ the \emph{conorm} with characteristic $v$ of $A$.
\end{definition}

Conorms naturally appear when tropicalizing the Schottky--Igusa modular form which cuts out the Schottky locus in genus $4$ \cite[\S4]{ChuaKummerSturmfels}. Conorms of matrices in the matroidal locus are of special interest:

\begin{proposition}\label{prop:conormsmatroidal}
 Assume that $A\in S^n_{>0}$ lies in the matroidal locus. Then $\vartheta_u(A)\geq0$ for all $u\in\parity_n$. More precisely, let $M$ be a {totally} unimodular $n\times r$ matrix with columns $v_1,\ldots,v_r\in\Z^n$, no two of which are linearly dependent, such that we can write $$A=\sum_{i=1}^r c_i\cdot v_iv_i^t$$ for some $c_i\geq0$. Then $\vartheta_u(A)=c_i$ if $2u\equiv v_i \mod 2$ and $\vartheta_u(A)=0$ if $2u$ is not congruent to any $v_i$.
\end{proposition}

\begin{proof}
 This was shown in the course of the proof of \cite[Thm.~4.2]{ChuaKummerSturmfels}.
\end{proof}

This implies the Conway--Sloane conjecture for matrices in the matroidal locus.

\begin{corollary}
 Let $A,B\in S^n_{>0}$ in the matroidal locus. If $\Theta(A)=\Theta(B)$, then there is $S\in\GL_n(\Z)$ such that $A=S^t B S$.
\end{corollary}

\begin{proof}
 Let $M$ be a {totally} unimodular $n\times r$ matrix with columns $v_1,\ldots,v_r\in\Z^n$, no two of which are linearly dependent, such that we can write $$A=\sum_{i=1}^r c_i\cdot v_iv_i^t$$ for some $c_i\geq0$. We first observe that Proposition \ref{prop:conormsmatroidal} implies that the regular matroid represented by $M$ can be recovered as the binary matroid represented by the $2u$ with $u\in\parity_n$ and $\vartheta_u(A)\neq0$. Thus it is uniquely determined by $\Theta(A)$. By \cite[Cor.~10.1.4]{oxley} every regular matroid is uniquely representable which implies that the rank one matrices $v_iv_i^t$ are uniquely determined by $\Theta(A)$ up to simultaneous $\GL_n(\Z)$-action. Finally, Proposition \ref{prop:conormsmatroidal} shows that the coefficients $c_i$ are also determined by $\Theta(A)$ which completes the proof.
\end{proof}

We conjecture that the matroidal locus can be characterized by the conorms.

\begin{conjecture}\label{con:matroidal}
A matrix $A\in S^n_{>0}$ lies in the matroidal locus if and only if $\vartheta_u(A)\geq0$ for all $u\in\parity_n$.
\end{conjecture}

Note that a proof of Conjecture \ref{con:matroidal} would also give a positive answer to \cite[Question 4.11]{ChuaKummerSturmfels}.
Since for $n\leq3$ every positive definite matrix is in the matroidal locus, the conjecture is true for $n\leq3$ by Proposition \ref{prop:conormsmatroidal}. We show that it is true for $n=4,5$ as well.

\begin{theorem}
Conjecture \ref{con:matroidal} is true for $n=4$.
\end{theorem}

\begin{proof}
 One direction is Proposition \ref{prop:conormsmatroidal}.
 For the other direction note that it follows from Lemma \ref{lem:linear} that $\vartheta_u$ is linear on the secondary cone of every Delaunay subdivision of $\R^n$. For $n=4$ we use the extreme rays $R_1,\ldots,R_{12}$ of the three secondary cones $\mathcal{D}_1,\mathcal{D}_2,\mathcal{D}_3$ given in \cite[\S 4.4.3]{VallentinThesis}. Every matrix from $\mathcal{D}_1$ is in the matroidal locus. The same is true for every matrix in a face of $\mathcal{D}_2$ and $\mathcal{D}_3$ that does not have $R_{11}$ as an extreme ray. Every other matrix in the secondary cone of $\mathcal{D}_2$ or $\mathcal{D}_3$ is the sum of a positive multiple of $R_{11}$ and a conic combination of matrices from $\{R_1,\ldots,R_{12}\}\setminus\{R_5\}$. Letting $u=(\frac{1}{2},\frac{1}{2},0,0)$ one calculates that $\vartheta_u(R_{11})=-1$ and $\vartheta_u(R_{k})=0$ for all $k\not\in\{5,11\}$. Therefore, by linearity we have $\vartheta_u(A)<0$ for every such matrix $A$.
\end{proof}

\begin{theorem}
Conjecture \ref{con:matroidal} is true for $n=5$.
\end{theorem}

\begin{proof}
On a given iso-edge domain $D$ and for a fixed vector $v\in \parity_n$ the expression $\vartheta_v(A)$ is linear in $A$.
We then consider the following polyhedral cone embedded in $D$:
\begin{equation*}
\left\{ A\in D \mbox{~and~} \vartheta_v(A)\geq 0 \mbox{~for~}v\in \parity_n \right\}
\end{equation*}
We can compute its description by using linear programming. We considered all the $76$ domains in
dimension $5$ and found out that for each of them this set is spanned by rank one matrices that determine
an {totally} unimodular system of vectors and so belong to the matroidal locus.
Furthermore, the $4$ maximal {totally} unimodular systems in dimension $5$ are obtained in this way which concludes {the proof}.
\end{proof}

\section{Generalizations}\label{SectionGeneralization}
The construction of iso-edge domains can be generalized to several different contexts.
We list here some that could be of interest for future {research}.

\subsection{Generalization to subspaces}
In \cite{MartinetSigrist,EquivariantLtype} the perfect form and secondary cones were generalized
to the case of matrices in a vector space ${\mathcal V}\subset S^n$ such
that ${\mathcal V}\cap S^n_{>0}\not=\emptyset$. This kind of setting is used in an enormous
number of cases for a multiplicity of purposes.
It would be interesting to extend this to the case of iso-edge domains.

\subsection{Generalization to $k$-cells}
The iso-edge domains concern the case of dimension $1$ cells of the Delaunay polytopes while the
secondary cones consider the full dimensional 
cells of the Delaunay tessellation.
If we fix the $k$-dimensional faces of Delaunay polytope then we also get a polyhedral decomposition
of $S^n_{>0}$.
{Obviously, the decomposition for $k=n$ would coincide with the one for $k=n-1$ but based on
  dimension $3$ one would expect it to also coincide with the one for $k=n-2$.
  If that was true, the first interesting case would be dimension $5$ where the cases $k=1$,
  $2$ and $5$ would likely be all different.}

\subsection{Generalization to the case of symmetric faces}
The iso-edge domains correspond to the centrally symmetric faces of the Delaunay tessellation.
If we take a finite subgroup $G$ of $\GL_n(\ZZ)$, then we can consider the forms invariant under this
group. The group $G$ acts on $\RR^n / \ZZ^n$. If the set of fixed points $S$ is finite, then we can consider
the set of closest points to them. Those define invariant cells of the Delaunay tessellation.
If $G = \{\pm Id_n\}$, then we get the case of iso-edge domains.

\section{Acknowledgments}
The first author thanks Viatcheslav Grishukhin for introducing him to the works of Ryshkov and Baranovskii and Klaus
Hulek for introducing him to the theory of compactifications.
Both authors thank Frank Vallentin for the invitation to the conference ``Discrete geometry with a view on
symplectic and tropical geometry''. {Finally, we thank the anonymous referee for comments that improved the manuscript.}

\bibliographystyle{amsplain_initials_eprint}
\bibliography{RefIsoEdge}

\end{document}